\documentclass[12pt,article]{amsart}
\usepackage[a4paper,textwidth=17.5cm,textheight=21cm,includehead]{geometry}

\usepackage[latin1]{inputenc}

\usepackage{latexsym}
\usepackage{amsthm}
\usepackage{epic}
\usepackage{amsfonts}

\usepackage{rotating}

\usepackage[T1]{fontenc}

\usepackage{amsmath} 
\usepackage{amssymb} 
\usepackage[all]{xy}  
\usepackage{graphicx}  
\usepackage{xcolor}
\usepackage{mathrsfs}
\usepackage{euscript}
\usepackage{listings}

\newtheorem{thm}{Theorem}[section]
\newtheorem{prop}[thm]{Proposition}

\newtheorem{cor}[thm]{Corollary}
\newtheorem{lem}[thm]{Lemma}
\newtheorem{defn}[thm]{Definition}

\newtheorem{remark}[thm]{Remark}
\newtheorem{example}[thm]{Example}

\newcommand{\R}{\mathbb{R}}

\newcommand{\C}{\mathbb{C}}

\newcommand{\Pe}{\mathbb{P}}

\makeatletter
      \def\@setcopyright{}
      \def\serieslogo@{}
      \makeatother

\begin{document}
\author{Giorgio Ottaviani, Alicia Tocino}
\address{Dipartimento di Matematica e Informatica ``Ulisse Dini'', University of Florence, Italy}
\email{ottavian@math.unifi.it, aliciatocinosanchez@ucm.es}

\title{\textbf{Best rank k approximation for binary forms}}

\begin{abstract}
In the tensor space $\mathrm{Sym}^d \R^2$ of binary forms we study the best rank $k$ approximation problem. The critical points of the best rank $1$ approximation problem are the eigenvectors and it is known that they span a hyperplane.
We prove that the critical points of the best rank $k$ approximation problem lie in the same hyperplane.
\end{abstract}

\maketitle

\section{Introduction}
The symmetric tensor space $\mathrm{Sym}^dV$, with $V=\R^2$ (resp. $V=\C^2$), contains real (resp. complex) binary forms,
which are homogeneous polynomials in two variables. The forms which can be written as $v^d$, with $v\in V$, correspond to polynomials which are the $d$-power of a linear form,
they have rank one.
We denote by $C_d\subset \mathrm{Sym}^dV$ the variety of forms of rank one.
The $k$-secant variety $\sigma_k(C_d)$ is the closure of the set of forms which can be written as $\sum_{i=1}^k\lambda_iv_i^d$ with $\lambda_i\in\R$ (resp. $\lambda_i\in\C$).

 We say that a nonzero rank $1$ tensor is a critical rank one tensor for $f\in \mathrm{Sym}^d V$ if it is a critical point of the distance function from $f$ to the variety of rank $1$ tensors. 
Critical rank one tensors are important to determine the best rank one approximation of $f$, in the setting of optimization
\cite{FriTam, Lim, Stu}. Critical rank one tensors may be written as $\lambda v^d$ with $\lambda\in\C$ and $v\cdot v=1$, the last scalar product is the Euclidean scalar product.
The corresponding vector $v\in V$ has been called tensor eigenvector, independently by Lim and Qi, \cite{Lim, Qi}. In this paper we concentrate on
critical rank one tensors $\lambda v^d$, which live in $\mathrm{Sym}^dV$ (not in $V$ like the eigenvectors), for a better comparison with critical rank $k$ tensors, see Definition \ref{def:kcritical} .

There are exactly $d$ critical rank one tensor (counting with multiplicities)
for any $f$ different from $c(x^2+y^2)^{d/2}$ (with $d$ even), while 
there are  infinitely many critical rank one tensors for $f=(x^2+y^2)^{d/2}$ (see Prop. \ref{prop:eigendisc}).

The critical rank one tensors for $f$ are contained in the hyperplane $H_f$
(called the singular space, see \cite{OP}), which is orthogonal to the vector $D(f)=yf_x-xf_y$. We review this statement at the beginning of \S \ref{sec:singularspace}.

The main result of this paper is the following extension of the previous statement to critical rank $k$ tensors, for any $k\ge 1$. 

\begin{thm}\label{thm:main} Let $f\in \mathrm{Sym}^d\C^2$ .\hfill

i)  All critical rank $k$ tensors for $f$ are contained in the hyperplane $H_f$, for any $k\ge 1$.

ii) Any  critical rank $k$ tensor for $f$  may be written as a linear combination of the  critical rank $1$ tensors for $f$.

\end{thm}

Theorem \ref{thm:main} follows after Theorem \ref{mainTheorem} and Proposition \ref{prop:main2}. Note that Theorem \ref{thm:main} may applied to the best rank $k$ approximation of $f$, which turns out to be contained in $H_f$
and may then be written as a linear combination of the  critical rank $1$ tensors for $f$. This statement may be seen as a weak extension of the Eckart-Young Theorem to tensors. Indeed, in the case of matrices, the best rank $k$ approximation is exactly the sum of the first $k$ critical rank one tensors, by the Eckart-Young Theorem, see \cite{OP}. The polynomial $f$ itself may be written as linear combination of its critical rank $1$ tensors, see Corollary \ref{cor:corf}, this statement may be seen as a {\it spectral decomposition for $f$}.  All these statements may be generalized to the larger class of tensors, not necessarily symmetric, in any dimension, see \cite{DOT}.

In \S\ref{sec:lastreal} we report about some numerical experiments regarding the number of real critical rank $2$ tensors in $\mathrm{Sym}^4\R^2$.

\section{Preliminaries}
Let $V=\R^2$ equipped with the Euclidean scalar product. 
The associated quadratic form has the coordinate expression $x^2+y^2$,
with respect to the orthonormal basis $x, y$.
The scalar product can be extended to a scalar product on the tensor space $\mathrm{Sym}^dV$ of binary forms,
which is $SO(V)$-invariant.
For powers $l^d$, $m^d$ where $l, m\in V$, we set
$\langle l^d, m^d\rangle : = \langle l, m\rangle^d$
and by linearity this defines the scalar product on the whole $\mathrm{Sym}^dV$ (see Lemma \ref{lema:scalarproduct}).

Denote as usual $\left\|{f}\right\|=\sqrt{\langle f, f\rangle}$.

For binary forms which split in the product of linear forms we have the formula
\begin{equation}\label{eq:decomp}\langle l_1l_2\cdots l_d, m_1m_2\cdots m_d\rangle =
\frac{1}{d!}\sum_{\sigma}\langle l_1,m_{\sigma(1)}\rangle \langle l_2,m_{\sigma(2)}\rangle 
\cdots \langle l_d,m_{\sigma(d)}\rangle \end{equation}

The powers $l^d$ are exactly the tensors of rank one in $\mathrm{Sym}^dV$,
they make a cone $C_d$ over the rational normal curve.

The sums $l_1^d+\ldots +l_k^d$ are the tensors of rank $\le k$, and equality holds when the number of summands is minimal. The closure of the set of tensors of rank $\le k$, both in the Euclidean or in the Zariski topology, is a cone $\sigma_kC_d$,
which is the $k$-secant variety of $C_d$.

The Euclidean distance function $d(f,g)=\left\|f-g\right\|$ is our objective function. 
The optimization problem we are interested is, given a real $f$, to minimize $d(f,g)$ with the
constraint that $g\in \left(\sigma_kC_d\right)_\R$. This is equivalent to minimize the square function $d^2(f,g)$,
which has the advantage to be algebraic. The number of complex critical points of the square distance function $d^2$
is called the Euclidean distance degree (EDdegree \cite{DHOST}) of $\sigma_kC_d$ and has been computed for small values of
$k, d$ in the rightmost chart in Table 4.1 of \cite{OSS}. We do not know a closed formula for these values,
although \cite[Theorem 3.7]{OSS} computes them in the case of a general quadratic distance function, not
$SO(2)$-invariant.

\section{Critical points of the distance function}

Let us recall the notion of eigenvector for symmetric tensors (see \cite{Lim, Qi},\cite[Theorem 4.4]{OP}). 
\begin{defn}\label{def:eigentensor}
Let $f\in \mathrm{Sym}^d V$. We say that a nonzero rank $1$ tensor is a  critical rank one tensor for $f$ if it is a critical point of the distance function from $f$ to the variety of rank $1$ tensors. It is convenient to write a critical rank one tensor in the form $\lambda v^d$ with $\left\|{v}\right\|=1$, in this way
$v$ is defined up to $d$-th roots of unity and is called an eigenvector of $f$ with eigenvalue $\lambda$.
\end{defn}

\begin{remark}
Let $d=2$ and let $f$ be a symmetric matrix.
 All the critical rank one tensors of $f$ have the form $\lambda v^2$ where $v$ is a classical eigenvector of norm $1$ for the symmetric matrix $f$, with eigenvalue $\lambda$. 
\end{remark}

\begin{lem}\label{lem:eigentensor}
Given $f\in\mathrm{Sym}^d V$, the point $\lambda v^d$ of rank $1$, with $\left\|{v}\right\|=1$, is a critical rank one tensor for $f$ if and only if $\langle f,v^{d-1}w\rangle=
\lambda \langle v,w\rangle$ $\forall w\in V$, which can be written (identifying $V$ with $V^\vee$ according to the Euclidean scalar product) as 
$$f\cdot v^{d-1}= \lambda v,$$ with $\lambda=\langle f, v^d\rangle  $.
\end{lem}
\begin{proof} The property of critical point is equivalent to $f-\lambda v^d$ being orthogonal to
$v^{d-1}w$ $\forall w\in V$, which gives
$\langle f, v^{d-1}w\rangle =\langle \lambda v^d,v^{d-1}w\rangle$
$\forall w\in V$. The right-hand side is $\left\|{v}\right\|^{2d-2} \lambda \langle v,w\rangle=\lambda \langle v,w\rangle$, as we wanted.
Setting $w=v$ we get $\langle f, v^{d}\rangle =\lambda$.
\end{proof}

On the other hand, eigenvectors correspond to critical points of the function
$f(x,y)$ restricted on the circle $S^1=\{(x,y)|x^2+y^2=1\}$ (\cite{Lim, Qi}).
By Lagrange multiplier method, we can compute the eigenvectors of $f$ as the normalized solutions $(x,y)$ of:
\begin{equation}\label{eq1}
\mathrm{rank} 
\begin{bmatrix}
f_{x} & f_{y}  \\
x & y  
\end{bmatrix} \leq 1
\end{equation}
This corresponds with the roots of discriminant polynomial $D(f)=yf_x-xf_y$. 
$D$ is a well known differential operator which satisfies the Leibniz rule, i.e. $D(fg)=D(f)g+fD(g)$
$\forall f, g\in \mathrm{Sym}^d V$.
For any $l=ax+by\in V$ denote $l^\perp=D(l)=-bx+ay$. Note that $\langle l,l^\perp\rangle =0$.

We have the following:
\begin{prop}\label{prop:eigendisc} Consider $f(x,y)\in\mathrm{Sym}^dV$:
\begin{itemize}
	\item If $v$ is eigenvector of $f$ then $D(v)=v^\perp$ is a linear factor of $D(f)$.
	\item Assume that $D(f)$ splits as product of distinct linear factors and $v^\perp|D(f)$, then $\frac{v}{\left\|{v}\right\|}$ is an eigenvector of $f$.
\end{itemize}
\end{prop}
We postpone the proof after Prop. \ref{prop:critical}.

Now let us differentiate some cases in terms of $D(f)$ (see Theorem $2.7$ of \cite{ASS}):
\begin{itemize}
	\item if $d$ is odd: $D(f)= 0$ if and only if $f= 0$, in particular $D:\mathrm{Sym}^dV\rightarrow \mathrm{Sym}^dV$ is an isomorphism.
	\item if $d$ is even: $D(f)=0$ if and only if $f=c(x^2+y^2)^{d/2}$ for some $c\in\R$. We will show in Lemma \ref{lemma2} which are the eigenvectors in this case. The image of $D:\mathrm{Sym}^dV\rightarrow \mathrm{Sym}^dV$ is the space orthogonal to $f=(x^2+y^2)^{d/2}$.
\end{itemize}

\begin{lem} (\cite{LS}, Section $2$)\label{lema:scalarproduct}
Suppose $f=\sum_{i=0}^{d} \binom{d}{i} a_i x^iy^{d-i}$ and $g=\sum_{i=0}^{d} \binom{d}{i} b_i x^iy^{d-i}$. Then we get:
\begin{equation}\label{eq:scalar}
\langle f,g\rangle:=\sum_{i=0}^{d} \binom{d}{i} a_ib_i
\end{equation}
where $\langle\,,\,\rangle$ is the scalar product defined in the introduction.
\end{lem}

\begin{proof}
By linearity we may assume $f=(\alpha x+\beta y)^d$ and $g=(\alpha'x+\beta' y)^d$. The right-hand side of (\ref{eq:scalar}) gives
$$\langle f,g\rangle=\sum_{i=0}^{d} \binom{d}{i}(\alpha\alpha')^i(\beta \beta')^{d-i}=(\alpha\alpha'+\beta\beta')^d$$
which agrees with $\langle \alpha x+\beta y,\alpha' x+\beta' y\rangle^d$.
\end{proof}

\begin{lem}\label{remark:2}
Let $f=(x^2+y^2)^{d/2}\in \mathrm{Sym}^d V$ with $d$ even, and $v=\alpha x+\beta y\in V$, $v\neq 0$,  then $\langle v^d,f\rangle=\left\|{v}\right\|^d$.
\end{lem}

\begin{proof}
By applying (\ref{eq:decomp}) with a grain of salt (e.g. decomposing $x^2+y^2$
into two conjugates linear factors) we get

$$\langle v^d,f\rangle=\langle (x^2+y^2),v^2\rangle^{d/2} = 
(\alpha^2+\beta^2)^{d/2}=\left\|{v}\right\|^d.$$
\end{proof}

\begin{lem}\label{lemma2}
If $f=(x^2+y^2)^{d/2}\in \mathrm{Sym}^d V$ then, for every  nonzero $v\in V$, 
$\langle f, v^{d-1}w\rangle=\left\|{v}\right\|^{d-2}\langle v, w\rangle$.
In particular every vector $v$ of norm $1$ is eigenvector of $f$ with eigenvalue $1$.

\end{lem}

\begin{proof}
As in Lemma \ref{remark:2} we get

$$\langle f, v^{d-1}w\rangle = {\langle (x^2+y^2),v^2\rangle}^{d/2-1}
\langle (x^2+y^2),vw\rangle = \left\|{v}\right\|^{d-2}\langle v, w\rangle.$$

The second part follows by putting $w=v$ and equating with Lemma \ref{remark:2}. We get
$\langle f, v^{d-1}w\rangle=\langle v^d,f\rangle\langle v, w\rangle$ just in the case $|v|=1$.

\end{proof}
\begin{remark}
Lemma \ref{lemma2} extends the fact that every vector of norm $1$ is eigenvector of the identity matrix with eigenvalue $1$.
The geometric interpretation of this lemma  
is that the $2$-dimensional cone
of rank $1$ degree $d$ binary forms cuts any sphere centered in $(x^2+y^2)^{d/2}$
in a curve. This curve
\end{remark}

\begin{lem}\label{lem}
The normal space at $l^d\in C_d$ coincides with $\left(l^\perp\right)^2\cdot \mathrm{Sym}^{d-2}V$
\end{lem}
\begin{proof}
The tangent space at $l^d$ is spanned by $l^{d-1}V$ and has dimension $2$. The elements in $\left(l^\perp\right)^2\cdot \mathrm{Sym}^{d-2}V$ are orthogonal to the tangent space,
moreover the dimension of this space is the expected one $d-1$.
\end{proof}

\begin{defn}\label{def:kcritical}
We say that $g\in\mathrm{Sym}^dV$ is a critical rank $k$ tensor for $f$ if it is a critical point of the distance function $d(f,\_)$ restricted on $\sigma_kC_d$.
\end{defn}

\begin{prop}\label{prop:critical} Let $2k\le d$. A polynomial $g=\sum_{i=1}^k \mu_i l_i^d \in \sigma_kC_d$ is a critical rank $k$ tensor for $f$ if and only if there exist $h\in\mathrm{Sym}^{d-2k}V$ such that
\begin{equation}\label{eq}
f=\sum_{i=1}^k \mu_i l_i^d +h\cdot \prod_{i=1}^k\left(l_i^\perp\right)^2
\end{equation}
\end{prop}

\begin{proof}
By Terracini Lemma, the tangent space of the point $g\in\sigma_kC_d$ is given by the sum of $k$ tangent spaces at $l_i^d=(a_ix+b_iy)^d$. By Lemma \ref{lem} the normal space of each of these tangent spaces are given by $\left(l_i^\perp\right)^2\cdot \mathrm{Sym}^{d-2}V$. Hence, the normal space to $g$ is given by intersection of the $k$ normal spaces, which is given by 
polynomials $\prod_{i=1}^k\left(l_i^\perp\right)^2 \cdot h$ where $h\in\mathrm{Sym}^{d-2k}V$.

Now suppose that $g$ is a critical rank $k$ tensor for $f$. This means that $f-g$ is in the normal space. Hence, $f-g$ is of the form $\prod_{i=1}^k\left(l_i^\perp\right)^2 \cdot h$ for some $h\in\mathrm{Sym}^{d-2k}V$.

Conversely, if $(\ref{eq})$ holds, we need that $f-g$ belongs to the normal space at $g$ which is also true by the construction of the normal space.
\end{proof}

\begin{proof} [Proof of Prop. \ref{prop:eigendisc}]
If $v$ is eigenvector of $f$ then $\langle f,v^d\rangle v^d$ is critical rank $1$ tensor for $f$ (by Lemma \ref{lem:eigentensor}). By Prop. \ref{prop:critical}
 $f=\langle f,v^d\rangle v^d+h \left(v^\perp\right)^2$ where $h\in\mathrm{Sym}^{d-2}V$. Applying the operator $D$ to $f$ we get by Leibniz rule, since $D(v)=v^{\perp}$ and $D(v^{\perp})=-v$:
$$D(f)=\langle f,v^d\rangle dv^{d-1}v^\perp+D(h)\left(v^\perp\right)^2-2vv^\perp h\Longrightarrow v^\perp|D(f)$$ 
Conversely, since we assume there are $d$ distinct eigenvectors, then we find all the linear factors of $D(f)$.
\end{proof}

This proposition is connected with Theorem $2.5$ of \cite{LS}.

\section{The singular space}\label{sec:singularspace}
In \cite{OP} it was considered the singular space $H_f$ as the hyperplane orthogonal to $D(f)=yf_x-xf_y$. It follows from Prop. \ref{prop:eigendisc}
that the critical rank $1$ tensor for $f$ belong to $H_f$ (since the eigenvectors of $f$ can be computed as the solutions of (\ref{eq1}) that coincides with $D(f)$ for binary forms), see \cite[Def. 5.3]{OP}.
It is worth to give a direct proof that the critical rank $1$ tensors for $f$ belong to $H_f$, the hyperplane orthogonal to $D(f)$,
 based on Prop. 
\ref{prop:critical}.

Let $\mu l^d$ be a critical rank $1$ tensors for $f$, then by Prop. 
\ref{prop:critical} there exist   $h\in\mathrm{Sym}^{d-2}V$ such that 
$f= \mu l^d +h\left(l^\perp\right)^2$.

We have to prove $\langle D(f), l^d\rangle =0$ which follows immediately from (\ref{eq:decomp})
since $l^{\perp}$ divides $D(f)$ by Prop. \ref{prop:eigendisc}.

\begin{lem}\label{lem:lmperp}
Let $l, m\in V$, Then $\langle l^\perp, m\rangle +\langle m^\perp, l\rangle =0$.
\end{lem}
\begin{proof}
Straightforward.
\end{proof}

Our main result is

\begin{thm}\label{mainTheorem}
The critical points of the form $\sum_{i=1}^{k}\mu_i l_i^d$ of the distance function $d(f,-)$ restricted on $\sigma_kC_d$
belong to $H_f$.
\end{thm}

\begin{proof}
Given a decomposition
$f= \sum_{i=1}^k \mu_i l_i^d +h\cdot \prod_{i=1}^k\left(l_i^\perp\right)^2$, with $h\in\mathrm{Sym}^{d-2k}V$,
we compute 
\begin{equation}\label{eq:3sum}
D(f)=d\sum_{i=1}^k \mu_il_i^{\perp}l_i^{d-1}-\sum_{i=1}^k2l_il_i^{\perp}\prod_{j\neq i}^k\left(l_j^\perp\right)^2h+D(h)\prod_{i=1}^k\left(l_i^\perp\right)^2\end{equation}
and we have to prove \begin{equation}\label{eq:dfli}\langle D(f),\sum_{j=1}^k l_j^d\rangle =0.\end{equation}
We compute separately the contribution of the three summands in (\ref{eq:3sum}) to the scalar product with $l_j^d$.

We have for the first summand
$$\langle \left(\sum_{i=1}^kl_i^{\perp}l_i^{d-1}\right), l_j^d\rangle = \sum_{i=1}^k\langle l_i^\perp, l_j\rangle\langle l_i\cdot l_j\rangle^{d-1}$$
Summing over $j$ we get zero by Lemma \ref{lem:lmperp}.

We have for the second summand
$$\langle\left(\sum_{i=1}^kl_i,l_i^{\perp}\prod_{p\neq i}^k\left(l_p^\perp\right)^2h\right),l_j^d \rangle =
\langle\left(l_jl_j^{\perp}\prod_{p\neq j}^k\left(l_p^\perp\right)^2h\right),l_j^d \rangle=0 $$

We have for the third summand
$$\langle\left(D(h)\prod_{i=1}^k\left(l_i^\perp\right)^2\right), l_j^d\rangle = 0$$
 
Summing up, this proves (\ref{eq:dfli}) and then the thesis.

\end{proof}

\begin{example}
If $f=x^3y+2y^4$ then there are $6$ critical points of the form $l_1^4+l_2^4$ and $x^3y$ which lies on the tangent line at $x^4$. It cannot be written as $l_1^4+l_2^4$ and indeed it has rank $4$.
\end{example}

\section{The scheme of eigenvectors for binary forms}
Suppose $f\in \mathrm{Sym}^d V$ a symmetric tensor and $\mathrm{dim} V=2$. We denote by $Z$ the scheme defined by the polynomial $D(f)$, embedded in $\Pe(\mathrm{Sym}^d V)$ by the $d$-Veronese embedding in $\Pe V$ (see \cite{AEKP} for the case of matrices).

\begin{prop}\label{prop:main2}
$\langle Z \rangle = H_f$.
\end{prop}
\begin{proof}
$(i)$ If $D(f)$ has $d$ distinct roots then it is known that $\langle Z \rangle\subseteq H_f$, since $H_f$ is the hyperplane orthogonal to $D(f)$ (Theorem \ref{mainTheorem} with $k=1$). Hence $\langle Z \rangle\subseteq H_f$. 

$(ii)$ Now let us suppose that $D(f)$ has multiple roots but $f\neq (x^2+y^2)^{d/2}$. We show that $\langle Z\rangle\subseteq H_f$ by a limit argument. For every tensor $f$ such that $f\neq 0$ and $f\neq(x^2+y^2)^{d/2}$ there exists a sequence $(f_n)$ such that $f_n\rightarrow f$ and $D(f_n)$ has distinct roots for all $n$. Then, $H_{f_n}\rightarrow H_f$ because the differential operator is continuous. Moreover $H(f_n)$ is a hyperplane for all $n$. On the other hand, by definition we have that $\langle Z_{f_n}\rangle$ is the spanned of the roots of $D(f_n)$. When $f_n$ goes to the limit we get that $\langle Z_{f_n}\rangle\rightarrow \langle Z\rangle$. Hence, $\langle Z\rangle\subseteq H_f$.
	
$(iii)$ In the case that $f=(x^2+y^2)^{d/2}$ with $d$ even, then by Lemma \ref{lemma2} we know that every unitary vector is an eigenvector and $H_f$ is the ambient space. Hence, $\langle Z\rangle=H_f$.

We prove now that $\mathrm{dim} \langle Z \rangle=\mathrm{dim} H_f$ for $(i)$ and $(ii)$. 
Since $\mathcal{I}_{Z,\Pe^1}=\mathcal{O}_{\Pe^1}(-d)$,
$$\mathrm{codim}\langle Z\rangle=h^0(\mathcal{I}_{Z,\Pe^1}(d))=h^0(\mathcal{O}_{\Pe^1}(-d+d))=h^0(\mathcal{O}_{\Pe^1})=1$$
which coincides with the codimension of $H_f$.

\def\niente{
First we suppose that $D(f)$ has $d$ distinct roots. 
It is known that $\langle Z \rangle\subseteq H_f$, since $H_f$ is the hyperplane orthogonal to $D(f)$.
We prove that the equality holds by showing that $\mathrm{dim} \langle Z \rangle=\mathrm{dim} H_f$.

If we embedded the $d$ distinct eigenvectors $v_1,\ldots,v_d$ into the rational normal curve of degree $d$, $C_d$, it turns out $d$ independent elements $v_1^d,\ldots,v_d^d$. The embedding of each of the eigenvectors in $C_d$ is of the form:
$$(1:v_i)\mapsto (1:v_i:v_i^2:\ldots:v_i^d)\quad i=1,\ldots,d$$
We know that these points are independent by using \textit{Vandermonde determinant} since
\begin{equation*}
\mathrm{det}
\begin{bmatrix}
1 & v_1 & v_1^2 & \ldots & v_1^d\\
1 & v_2 & v_2^2 & \ldots & v_2^d\\
\vdots & \vdots & \vdots & \ddots & \vdots\\
1 & v_d & v_d^2 & \ldots & v_d^d
\end{bmatrix}
=\prod_{i<j}(v_j-v_i)\neq 0
\end{equation*}
Hence, the dimension of each of the spaces are equal.

Let us suppose that $D(f)$ has multiple roots but $f\neq (x^2+y^2)^{d/2}$. First we show that $\langle Z\rangle\subseteq H_f$ by a limit argument. For every tensor $f$ such that $f\neq 0$ and $f\neq(x^2+y^2)^{d/2}$ there exists a sequence $(f_n)$ such that $f_n\rightarrow f$ and $D(f_n)$ has distinct roots for all $n$. Then, $H_{f_n}\rightarrow H_f$ because the differential operator is continuous. Moreover $H(f_n)$ is a hyperplane for all $n$. On the other hand, by definition we have that $\langle Z_{f_n}\rangle$ is the spanned of the roots of $D(f_n)$. When $f_n$ goes to the limit we get that $\langle Z_{f_n}\rangle\rightarrow \langle Z\rangle$. Hence, $\langle Z\rangle\subseteq H_f$.

Now we prove the other inclusion. Suppose that we have $r$ distinct eigenvectors $v_1,\ldots,v_r$ with multiplicities $m_1,\ldots,m_r$ and $m_1+\ldots +m_r=d$. 
If we embedded the $r$ distinct eigenvectors into the rational normal curve of degree $d$, $C_d$, it turns out $r$ independent elements $v_1^d,\ldots,v_r^d$. The embedding of each of the eigenvectors in $C_d$ is of the form:
$$(1:v_i)\mapsto (1:v_i:v_i^2:\ldots:v_i^{d-1}:v_i^d)\quad i=1,\ldots,r$$
We consider also its derivatives:
$$(1:v_i)\mapsto (0:1:2v_i:\ldots:(d-1)v_i^{d-2}:dv_i^{d-1})$$
$$(1:v_i)\mapsto (0:0:2:\ldots:(d-1)(d-2)v_i^{d-3}:d (d-1)v_i^{d-2})$$
$$\vdots$$
$$(1:v_i)\mapsto (0:0:0:\ldots:\binom{d-1}{m_i-1}v_i^{d-m_i}:\binom{d}{m_i-1}v_i^{d+1-m_i})$$
We know that these points are independent by using the \textit{Confluent Vandermonde determinant} since,
\begin{equation*}
\mathrm{det}
\begin{bmatrix}
1 & v_1 & v_1^2 & v_1^3 & \ldots & v_1^{d-1} & v_1^d\\
0 & 1 & 2v_1 & 3 v_1^2 & \ldots & (d-1)v_1^{d-2} & dv_1^{d-1}\\
0 & 0 & 2 & 6 v_1 & \ldots & (d-1)(d-2)v_1^{d-3} & d (d-1)v_1^{d-2}\\
\vdots & \vdots & \vdots & \ddots & \vdots\\
1 & v_2 & v_2^2 & v_2^3 & \ldots & v_2^{d-1} & v_2^d\\
0 & 1 & 2v_2 & 3 v_2^2 & \ldots & (d-1)v_2^{d-2} & dv_2^{d-1}\\
0 & 0 & 2 & 6 v_2 & \ldots & (d-1)(d-2)v_2^{d-3} & d (d-1)v_2^{d-2}\\
\vdots & \vdots & \vdots & \ddots & \vdots\\
1 & v_r & v_r^2 & v_r^3 & \ldots & v_r^{d-1} & v_r^d\\
0 & 1 & 2v_r & 3 v_r^2 & \ldots & (d-1)v_r^{d-2} & dv_r^{d-1}\\
0 & 0 & 2 & 6 v_r & \ldots & (d-1)(d-2)v_r^{d-3} & d (d-1)v_r^{d-2}\\
\vdots & \vdots & \vdots & \ddots & \vdots
\end{bmatrix}
=\prod_{1\leq i<j\leq r}(v_j-v_i)^{m_i m_j}\neq 0
\end{equation*}
Hence, the dimension of each of the spaces are equal.

Finally, if $f=(x^2+y^2)^{d/2}$ with $d$ even, then by Lemma \ref{lemma2} we know that every nonzero vector is an eigenvector and $H_f$ is the ambient space. Hence, $\langle Z\rangle=H_f$.}
\end{proof}

As a consequence we obtain the following corollary, which may be seen
as a {\it Spectral Decomposition} of any binary form $f$.

\begin{cor}\label{cor:corf}
Any binary form $f\in \mathrm{Sym}^d V$ with $\mathrm{dim} V=2$ can be written as a linear combination of the critical rank one tensors for $f$.
\end{cor}

The previous statement holds even in the special case $d$ even
and $f=(x^2+y^2)^{d/2}$, since from \cite[Theorem 9.5]{Rez} there exists $c_d\in{\mathbb R}$ such that the following decomposition holds $\forall\phi\in{\mathbb R}$
$$(x^2+y^2)^{d/2}=c_d\sum_{k=0}^{d/2}\left[\cos(\frac{2k\pi}{d+2}+\phi)x+\sin(\frac{2k\pi}{d+2}+\phi)y\right]^d$$

In this decomposition the summands on the right-hand side correspond to 
$(d+2)/2$ consecutive vertices of a regular $(d+2)$-gon.

In the $d=2$ case, the Spectral Theorem asserts any binary quadratic form $f\in\mathrm{Sym}^2\R^2$
can be written as sum of its rank one critical tensors. This statement fails for $d\ge 3$, as it can be checked already on the examples $f=x^d+y^d$ for $d\ge 3$,
where only two among the $d$ rank one critical tensors are used, namely $x^d$ and $y^d$, and the coefficients of the remaining $d-2$ rank one critical tensors 
in the Spectral Decomposition of $f$ are zero.

\section{Real critical rank $2$ tensors for binary quartics}
\label{sec:lastreal}

We recall the following result by M. Maccioni.

\begin{thm}\label{thm:maccioni}(Maccioni, \cite[Theorem 1]{Mac})
Let $f$ be a binary form.
$$\# \text{ real roots of f }\leq\# \text{ real critical rank 1 tensors for\ } f$$
The inequality is sharp, moreover it is the only constraint between the number of real roots and the number of real critical rank $1$  tensors, beyond parity mod $2$.
\end{thm}

As a consequence, as it was first proved in \cite{ASS}, hyperbolic binary forms (i.e. with only real roots) have  all real critical rank $1$  tensors.

We attempted to extend Theorem \ref{thm:maccioni} to rank $2$ critical tensors.
Our description is not yet complete and we report about some numerical experiments in the space $\mathrm{Sym}^4\R^2$. From these experiments it seems that the constraints
about the number of real rank $2$ critical tensors are weaker than for rank $1$ critical tensors.

For quartic binary forms the computation of the critical rank $2$ tensors 
is
easier since the dual variety of the secant variety $\sigma_2(C_4)$ is given by quartics which
are squares, which make a smooth variety.

The number of complex critical  rank $2$ tensors for a general binary form of degree d
was guessed in   \cite{OSS} to be $3/2d^2-9d/2+1$. For $d=4$
this number is $7$, which can be confirmed by a symbolic computation
on a rational random binary quartic.

In conclusion, for a general binary quartic there are $4$ complex critical  rank $1$ tensors
 and $7$
complex rank $2$ critical tensors.

The following table reports some computation done for the case of binary quartic forms, by testing several different quartics. The appearance of ``yes'' in the last column means that we have found an example
of a binary quartic with the prescribed number of distinct and simple
real roots, real rank $1$ critical tensors and real critical rank $2$  tensors.
Note that we have not found any quartic with the maximum number of seven real 
rank $2$ critical tensors, we wonder if they exist.
\begin{center}
\begin{tabular}{c | c| c |c| c}
&\#\text{real roots}& \#\text{real  critical rank 1 tensors} & \#\text{real critical rank 2
 tensors} & \\
\hline
 & $0$ & $2$ & $3$ & yes\\
 & $2$ & $2$ & $3$ & yes\\
 & $0$ & $2$ & $5$ & yes\\
 & $2$ & $2$ & $5$ & yes\\
 & $0$ & $4$ & $3$ & yes\\
 & $2$ & $4$ & $3$ & yes\\
 & $4$ & $4$ & $3$ & yes\\
 & $0$ & $4$ & $5$ & yes\\
 & $2$ & $4$ & $5$ & yes\\
 & $4$ & $4$ & $5$ & ?\\  
& * & * & 7 & ?     
\end{tabular}
\end{center}

\section{Acknowledgement} Giorgio Ottaviani is member of INDAM-GNSAGA. This paper has been partially supported by the Strategic Project ``Azioni di Gruppi su variet\'a e tensori'' of the University of Florence.


\begin{thebibliography}{999}
\bibitem{AEKP} H. Abo, D. Eklund, T. Kahle, C. Peterson,
Eigenschemes and the Jordan canonical form, Linear Algebra Appl. 496 (2016), 121-151.


\bibitem{ASS} H.  Abo,  A.  Seigal,  B.  Sturmfels,
Eigenconfigurations of Tensors,
to appear in Algebraic and Geometric Methods in Discrete Mathematics, arXiv:1505.05729v2.


\bibitem{DHOST} J. Draisma, E. Horobe\c{t}, G. Ottaviani, B. Sturmfels,
R. Thomas, The Euclidean Distance Degree of an Algebraic Variety, Found. Comput.
Math. 16 (2016), no. 1, 99-149

\bibitem{DOT} J. Draisma, G. Ottaviani, A. Tocino, Best rank $k$ approximation for tensors, in preparation, 2017.

\bibitem{FriTam} S. Friedland, V. Tammali,  Low-rank approximation of tensors,  {\it Numerical algebra, matrix theory, differential-algebraic equations and control theory}, 377--411, Springer, Cham, 2015. 

\bibitem{LS}
H. Lee, B. Sturmfels, Duality of Multiple Root Loci,  Journal of Algebra 446 (2016) 499-526. 


\bibitem{Lim} L.H. Lim,
Singular values and eigenvalues of tensors: a variational approach,
 Proc. IEEE Internat. Workshop on Comput. Advances in Multi-
Sensor Adaptive Processing (CAMSAP 2005), 129-132

\bibitem{Mac} M. Maccioni, The number of real eigenvectors of a real polynomial,
arXiv:1606.04737, to appear in Boll. Unione Matematica Italiana.


\bibitem{OP} G. Ottaviani, R. Paoletti,
A Geometric Perspective on the Singular
Value Decomposition, Rend. Istit. Mat. Univ. Trieste, Volume 47, 107--125, 2015.

\bibitem{OSS} G. Ottaviani, P.J. Spaenlehauer, B. Sturmfels, Exact solutions in structured low-rank approximation, 
SIAM Journal on Matrix Analysis and Applications, 35 (4) (2014), 1521--1542.

\bibitem{Rez} B. Reznick, Sums of even powers of real linear forms, Mem. Amer. Math. Soc. 96 (1992),
no. 463.

\bibitem{Qi} L. Qi,
Eigenvalues of a real supersymmetric tensor, J. of Symbolic Comput. 40 (2005), 1302-132

\bibitem{SS} A. Seigal, B. Sturmfels, Real rank two geometry,  Journal of Algebra 484 (2017) 310-333.



\bibitem{Stu} B. Sturmfels, Tensors and their Eigenvectors, Notices of the American Mathematical Society 63 (2016) 604--606. 

\end{thebibliography}
\end{document}